\documentclass[11pt]{article}
\usepackage{amsmath,amsthm}
\usepackage{amssymb}
\usepackage{graphics}
\usepackage{epsfig}
\usepackage{psfrag}
\usepackage[T1]{fontenc}
\usepackage[utf8]{inputenc}

\advance \topmargin by -\headheight
\advance \topmargin by -\headsep     
\evensidemargin \oddsidemargin
\newtheorem{prob}{Problem}[section]
\newtheorem{prop}[prob]{Proposition}
\newtheorem{theo}[prob]{Theorem}

\newtheorem{defi}[prob]{Definition}

\newtheorem{lem}[prob]{Lemma}

\begin{document}

\title{Perfect graphs of arbitrarily large clique-chromatic number} 


\author{Pierre Charbit\thanks{LIAFA - Universit\'e Paris Diderot - 
Paris 7 - Case 7014 - F-75205 Paris Cedex 13. Email: \texttt{pierre.charbit@liafa.univ-paris-diderot.fr}.}
\and Irena Penev \thanks{Universit\'e de Lyon,  LIP, UMR 5668, ENS Lyon - CNRS - 
UCBL - INRIA, France.
   Email: \texttt{irena.penev@ens-lyon.fr,nicolas.trotignon@ens-lyon.fr,stephan.thomasse@ens-lyon.fr}. 
Partially 
supported by the ANR Project \textsc{Stint} under \textsc{Contract ANR-13-BS02-0007} and by the LABEX MILYON (ANR-10-LABX-0070) of Universit\'e de
    Lyon, within the program ‘‘Investissements d'Avenir’’
    (ANR-11-IDEX-0007) operated by the French National Research Agency
    (ANR).}
\and St\'ephan Thomass\'e \footnotemark[2]
  \and Nicolas Trotignon \footnotemark[2]
}

\date{}
\maketitle

\begin{abstract} We prove that there exist perfect graphs of arbitrarily large clique-chromatic number. 
These graphs can be obtained from cobipartite graphs by repeatedly gluing along cliques. This negatively
answers a question raised by Duffus, Sands, Sauer, and Woodrow in~[Two-coloring all two-element maximal antichains, {\em J. Combinatorial Theory, Ser. A}, 57 (1991), 109--116].
\end{abstract} 

\section{Introduction}

All graphs in this paper are simple, finite, and non-null. 
A {\em clique} of a graph $G$ is a (possibly empty) set of pairwise adjacent vertices of $G$. A {\em clique-coloring} of a graph $G$ is an assignment of colors to the vertices of $G$ in such a way that no inclusion-wise maximal clique of size at least two of $G$ is monochromatic (as usual, a set of vertices is {\em monochromatic} if all vertices in the set received the same color). A {\em $k$-clique-coloring} of $G$ is a clique-coloring $\varphi:V(G) \rightarrow \{1,\ldots,k\}$ of $G$. $G$ is {\em $k$-clique-colorable} if there exists a $k$-clique-coloring of $G$. The {\em clique-chromatic number} of $G$, denoted by $\chi_C(G)$, is the smallest integer $k$ such that $G$ is $k$-clique-colorable. Note that every proper coloring of $G$ is also a clique-coloring of $G$, and so $\chi_C(G) \leq \chi(G)$. Furthermore, if $G$ is triangle-free, then $\chi_C(G) = \chi(G)$ (since there are triangle-free graphs of arbitrarily large chromatic number~\cite{mycielski:color, zykov}, this implies that there are triangle-free graphs of arbitrarily large clique-chromatic number). However, if $G$ contains triangles, $\chi_C(G)$ may be much smaller than $\chi(G)$. For instance, if $G$ contains a dominating vertex, then $\chi_C(G) \leq 2$ (we assign the color $1$ to the dominating vertex and the color $2$ to all other vertices of $G$), while $\chi(G)$ may be arbitrarily large. Note that this implies that the clique-chromatic number is not monotone with respect to induced subgraphs, that is, there exist graphs $H$ and $G$ such that $H$ is an induced subgraph of $G$, but $\chi_C(H) > \chi_C(G)$. (In particular, the restriction of a clique-coloring of $G$ to an induced subgraph $H$ of $G$ need not be a clique-coloring of $H$.) 

It was shown in~\cite{GravierHM03} that for any graph $H$, the class of graphs that do not contain $H$ as an induced subgraph has a bounded clique-chromatic number if and only if all components of $H$ are paths. The {\em clique number} of a graph $G$, denoted by $\omega(G)$, is the maximum size of a clique of $G$. A graph $G$ is {\em perfect} if all its induced subgraphs $H$ satisfy $\chi(H) = \omega(H)$. It was asked in \cite{DSSW} whether perfect graphs have a bounded clique-chromatic number. It has since been shown that graphs from many sublasses of the class of perfect graphs are $2$- or $3$-clique-colorable~\cite{tuza, bacso, DiamondFreePerfect, defossez, DSSW, MS, NoBalSkewPartCliqueCol}. There are well-known examples of perfect graphs of clique-chromatic number three (one example is the graph obtained from the cycle of length nine by choosing three evenly spaced vertices and adding edges between them), but until now, it was not known whether there were any perfect graphs of clique-chromatic number greater than three. The main result of the present paper is the following theorem. 
\begin{theo}\label{thm-main} 
There exist perfect graphs of arbitrarily large clique-chromatic number.
\end{theo}

Thus, the question from~\cite{DSSW} mentioned above has a negative answer. We prove Theorem~\ref{thm-main} by exhibiting, for each integer $k \geq 2$, a perfect graph $G_k$ of clique-chromatic number $k+1$. The graph $G_k$ is obtained from cobipartite graphs (i.e. complements of bipartite graphs) by repeatedly applying the operation of gluing graphs along a clique. The fact that $G_k$ is perfect follows from the fact that cobipartite graphs are perfect, together with the fact that the operation of gluing along a clique preserves perfection (that is, if two perfect graphs are glued along a clique, then the resulting graph is also perfect). Note also that it is immediate from the construction that $G_k$ does not contain any induced cycle of length at least five; furthermore, $G_k$ does not contain the complement of any odd cycle of length at least five as an induced subgraph. 

A {\em hereditary class} is a class of graphs that is closed under taking induced subgraphs, and a {\em clique-cutset} of a graph is a (possibly empty) clique whose deletion from the graph yields a disconnected graph. It was asked in~\cite{NoBalSkewPartCliqueCol} whether, if $c$ is a positive integer and $\mathcal{G}$ is a hereditary class such that every graph in $\mathcal{G}$ is either $c$-clique-colorable or admits a clique-cutset, there must exist a positive integer $d$ such that every graph in $\mathcal{G}$ is $d$-clique-colorable. Our construction of the family $\{G_k\}_{k=2}^{\infty}$ implies that this question has a negative answer (even if we restrict our attention to the case when all graphs in the class $\mathcal{G}$ are perfect). Indeed, let $\mathcal{G}$ be the class of all induced subgraphs of the graphs $G_k$ (with $k \geq 2$). Then $\mathcal{G}$ is a hereditary class (each of whose members is a perfect graph), and every graph in $\mathcal{G}$ is either cobipartite (and therefore $2$-clique-colorable~\cite{NoBalSkewPartCliqueCol}) or admits a clique-cutset. However, $\mathcal{G}$ contains graphs of arbitrarily large clique-chromatic number (because $G_k \in \mathcal{G}$ and $\chi_C(G_k) = k+1$ for each $k \geq 2$).

\section{Proof of Theorem~\ref{thm-main}}

If $n$ is a positive integer, we denote by $[n]$ the set $\{1,2,\ldots,n\}$. When $X$ is a 
set and $n$ a non-negative integer, we denote by ${X \choose n}$ the set of all $n$-element 
subsets of $X$. A {\em cobipartite graph} is a graph whose complement is bipartite. A {\em bipartition} of 
a cobipartite graph is a partition $(A,B)$ of its vertex-set such that $A$ and $B$ 
are both (possibly empty) cliques. If $G$ is a graph, $v \in V(G)$, and $X \subseteq V(G) \smallsetminus \{v\}$, 
we say that $v$ is {\em complete} (resp. {\em anti-complete}) to $X$ in $G$ provided that $v$ is adjacent 
(resp. non-adjacent) to all vertices of $X$. For disjoint sets 
$A,B \subseteq V(G)$, we say that $A$ is {\em complete} (resp. {\em anti-complete}) 
to $B$ in $G$ provided that every vertex of $A$ is complete (resp. anti-complete) to $B$ in $G$. 

Here is the crucial gadget that we will use in our construction.  

\begin{defi}
Let $n$ and $k$ be positive integers such that $n \geq k$, let $N:={n\choose k}$, and let $G$ be a graph. 
\begin{itemize}
\item
Let $C$ be an $N$-vertex clique of $G$; say $C = \{c_1,\ldots,c_N\}$. (Note that the number of bijections between $C$ and ${[n] \choose k}$ is $N!$.) The {\em $(n,k)$-expansion of $G$ at $C$}, denoted by $G_C^{n,k}$, is the graph obtained from $G$ by adding $N!$ new cliques (each of size $n$, and each associated with a bijection from $C$ to ${[n] \choose k}$), pairwise anti-complete to each other. The clique associated with the bijection $\phi:C \rightarrow {[n] \choose k}$ is denoted by $X_C^\phi = \{x_1^\phi,\ldots,x_n^\phi\}$. There are no edges between $X_C^\phi$ and $V(G) \smallsetminus C$, and for all $i \in [N]$ and $j \in [n]$, $c_i$ is adjacent to $x_j^\phi$ if and only if $j \in \phi(i)$. The {\em petal} of $G_C^{n,k}$ associated with $\phi$ is the ordered pair $(C,X_C^\phi)$; clearly, $G_C^{n,k}[C \cup X_C^\phi]$ is a cobipartite graph with bipartition $(C,X_C^\phi)$. 
\item The {\em universal $(n,k)$-expansion} of $G$ is the graph obtained from $G$ by performing the $(n,k)$-expansion of $G$ at every $N$-vertex clique of $G$. A {\em petal} of the universal $(n,k)$-expansion of $G$ is a petal of an $(n,k)$-expansion of $G$ at some $N$-vertex clique of $G$. 
\end{itemize}
\end{defi}

Note that the $(n,k)$-expansion of a clique $C = \{c_1,\ldots,c_N\}$ can be understood as the creation of all possible enumerations of ${X \choose k}$ (with $X = \{x_1,\ldots,x_n\}$), where the vertex $c_i$'s neighborhood in $X$ is precisely the $i$-th $k$-element subset of $X$. We take all resulting cobipartite graphs, and we glue them along $C$; we then glue the resulting graph and the graph $G$ that we started with along the clique $C$, and we thus obtain the $(n,k)$-expansion of $G$ at $C$.  

It is well-known (and easy to prove) that the operation of gluing along a clique preserves perfection (that is, if two perfect graphs are glued along a clique, then the resulting graph is also perfect). It is also well-known that cobipartite graphs are perfect. Since the universal $(n,k)$-expansion of a graph $G$ can be obtained from $G$ by sequentially gluing cobipartite graphs along cliques of $G$, it is easy to see that if $G$ is perfect, then so is its universal $(n,k)$-expansion. We state this below for future reference. 

\begin{prop} \label{universal-perfect} Let $n$ and $k$ be positive integers such that $n \geq k$. Then universal $(n,k)$-expansions preserve perfection, that is, if $G$ is a perfect graph, then the universal $(n,k)$-expansion of $G$ is also perfect. 
\end{prop} 

From now on, $k \geq 2$ is a fixed integer. To prove Theorem~\ref{thm-main}, we construct a perfect graph $G_k$ of clique-chromatic number $k+1$. (Note that for the purposes of proving Theorem~\ref{thm-main}, it is enough to show that our perfect graph $G_k$ satisfies $\chi_C(G_{k+1}) \geq k+1$, but for the sake of completeness, we prove that equality holds.) Our graph $G_k$ is obtained from a large complete graph by applying universal expansions several times. (Since complete graphs are perfect, Proposition~\ref{universal-perfect} guarantees that the graph $G_k$ that we obtain is also perfect.) Formally, we define a sequence $\{n_i\}_{i=0}^k$ by setting 
\begin{itemize}
 \item $n_k=(k+1)!$
\item $n_{i-1}={n_i^2 \choose n_i}$ for each $i \in [k]$ 
\end{itemize}
Since $n_{i-1}={n_i^2 \choose n_i}=n_i {n_i^2-1 \choose n_i-1}$, we have that $n_{i}$ divides $n_{i-1}$, and consequently, $(k+1)!$ divides $n_i$ for all $i$; this will be of use later in the proof. We also observe that $\{n_i\}_{i=0}^k$ is a strictly decreasing sequence. 

Now, let $H_0$ be a complete graph on $(k+1)n_0$ vertices, and for each $i \in [k]$, let $H_i$ be the universal $(n_{i}^2,n_{i})$-expansion of $H_{i-1}$. Finally, let $G_k=H_{k}$. 

\begin{prop} \label{Gk-perfect} The graph $G_k$ is perfect. 
\end{prop} 
\begin{proof} 
This is immediate from the construction of $G_k$ and Proposition~\ref{universal-perfect}. 
\end{proof}

Our goal is to show that $\chi_C(G_k) = k+1$. We first show that $\chi_C(G_k) \leq k+1$ by exhibiting a $(k+1)$-clique-coloring of $G_k$ (see Proposition~\ref{prop-leq-k+1}). We then show that $\chi_C(G_k) = k+1$ by proving that every $(k+1)$-clique-coloring of $G_k$ uses all $k+1$ colors, and consequently, no $k$-clique-coloring of $G_k$ exists (see Proposition~\ref{prop-geq-k+1}). 

\begin{prop} \label{prop-leq-k+1} $G_k$ is $(k+1)$-clique-colorable. 
\end{prop} 
\begin{proof} 
We proceed as follows: we exhibit a $2$-clique-coloring of $H_0$, we then extend that to a $3$-clique-coloring of $H_1$, and then we extend that to a $3$-clique-coloring of $H_2$. If $k = 2$, then we are done, and otherwise, we show that for each $i \in \{2,\ldots,k-1\}$, any $(i+1)$-clique-coloring of $H_i$ can be extended to an $(i+2)$-clique-coloring of $H_{i+1}$. 

First, assign the color $1$ to a single vertex (call it $x_0$) of the complete graph $H_0$, and assign the color $2$ to all other vertices of $H_0$; clearly, this is a $2$-clique-coloring of $H_0$. Next, for each component of $H_1 \smallsetminus V(H_0)$, choose a vertex that is non-adjacent to $x_0$ and color it $1$ (such a vertex exists because no vertex of $H_0$ has more than $n_1$ neighbors in any component of $H_1 \smallsetminus V(H_0)$, and each component of $H_1 \smallsetminus V(H_0)$ is a clique of size $n_1^2 > n_1$), and assign the color $3$ to all other vertices of the component. Clearly, this is a $3$-clique-coloring of $H_1$. Furthermore, note that the set of all vertices of $H_1$ colored $1$ is a stable set (and so no clique of $H_1$ contains more than one vertex colored $1$). 

We now extend our $3$-clique-coloring of $H_1$ to a $3$-clique-coloring of $H_2$. Let $X$ be the vertex-set of a component of $H_2 \smallsetminus V(H_1)$ (thus, $X$ is a clique of size $n_2^2$ of $H_2$), and let $C$ be the set of all vertices of $H_1$ that have a neighbor in $X$; thus, $(C,X)$ is a petal of $H_2$. In particular, $C$ is a clique of $H_1$, and consequently, at most one vertex of $C$ was assigned the color $1$. If exactly one vertex of $C$ was colored $1$, then let $x \in X$ be some neighbor of that vertex, and otherwise, let $x$ be any vertex of $X$. If all neighbors of $x$ in $C$ were colored $2$, assign to $x$ the color $3$, and otherwise, assign to $x$ the color $2$. Finally, assign the color $1$ to all vertices in $X \smallsetminus \{x\}$. We do this for each component of $H_2 \smallsetminus V(H_1)$, and the coloring of $H_2$ that we obtain is clearly a $3$-clique-coloring of $H_2$. If $k = 2$, then we are done. So suppose that $k \geq 3$.

Now, fix $i \in \{2,...,k-1\}$, and assume inductively that we have $(i+1)$-clique-colored $H_i$. We must $(i+2)$-clique-color $H_{i+1}$. Let $X$ be the vertex-set of a component of $H_{i+1} \smallsetminus V(H_i)$; we color $X$ as follows. We first pick any vertex $x$ of $X$, and if all neighbors of $x$ in $H_i$ (note that these vertices form a clique) were colored $1$, then we assign to $x$ the color $2$; otherwise, we assign to $x$ the color $1$. To all other vertices of $X$, we assign the color $i+2$. We do this for all components of $H_{i+1} \smallsetminus V(H_i)$, and we thus obtain an $(i+2)$-clique-coloring of $H_{i+1}$. This completes the induction, and it follows that $G_k = H_k$ is $(k+1)$-clique-colorable. 
\end{proof} 

We now need some notation. If $v \in V(G_k)$, we denote by $N(v)$ the set of all neighbors of $v$ in $G_k$ (in particular, $v \notin N(v)$). If $v \in V(G_k)$ and $X \subseteq V(G_k)$, we denote by $N_X(v)$ the set of all neighbors of $v$ in $X$, that is, $N_X(v) = N(v) \cap X$ (here, $v$ may or may not belong to $X$, but $v \notin N_X(v)$). The following proposition is easy but crucial for what follows. 

\begin{prop} \label{prop-max-clique} 
Let $i \in \{0,\ldots,k-1\}$, let $C$ be an $n_i$-vertex clique of $H_i$ (recall that $n_i = {n_{i+1}^2 \choose n_{i+1}}$), and let $c \in C$. For every petal $(C,X)$ in $H_{i+1}$, $\{c\} \cup N_X(c)$ is a maximal clique (of size $n_{i+1}+1$) of $G_k$. 
\end{prop}
\begin{proof} 
By definition, $\{c\}\cup N_X(c)$ is a clique of size $n_{i+1}+1$ of $H_{i+1}$ (and therefore of $G_k$ as well), and we just need to show that the clique $\{c\} \cup N_X(c)$ is maximal in $G_k$. We first show that it is maximal in $H_{i+1}$. We know that no vertex in $V(H_{i+1}) \smallsetminus (C \cup X)$ has a neighbor in $X$, and so we just need to show that no vertex in $C \smallsetminus \{c\}$ is complete to $N_X(c)$. But this follows immediately from the fact that vertices in $C$ have pairwise distinct neighborhoods in $X$, and all these neighborhoods are of equal size. Thus, $\{c\} \cup N_X(c)$ is a maximal clique of $H_{i+1}$. 

To show that $\{c\} \cup N_X(c)$ is a maximal clique of $G_k$, it now suffices to show that no vertex 
in $V(G_k) \smallsetminus V(H_{i+1})$ has more than $n_{i+1}$ neighbors in $V(H_{i+1})$ (this is sufficient 
because $\{c\} \cup N_X(c) \subseteq V(H_{i+1})$ and $|\{c\} \cup N_X(c)| = n_{i+1}+1$). If $i = k-1$, then 
this is immediate, and so we suppose that $i \leq k-2$. Fix $v \in V(G_k) \smallsetminus V(H_{i+1})$, and 
let $j \in \{i+1,\ldots,k-1\}$ be minimal such that $v \in V(H_{j+1})$. Fix a petal $(C',X')$ of $H_{j+1}$ 
such that $v \in X'$. Since $|X'| = n_{j+1}^2$, the number of $n_{j+1}$-element subsets of $X'$ that contain 
$v$ is precisely ${n_{j+1}^2-1 \choose n_{j+1}-1}$, and consequently, 
$|N_{C'}(v)| = {n_{j+1}^2-1 \choose n_{j+1}-1} \leq {n_{j+1}^2 \choose n_{j+1}} = n_j \leq n_{i+1}$. 
Since $N_{V(H_{i+1})}(v) \subseteq N_{V(H_j)}(v) = N_{C'}(v)$, it follows that $v$ has at most $n_{i+1}$ 
neighbors in $H_{i+1}$, which is what we needed to show. 
\end{proof} 

From now on, we fix a $(k+1)$-clique-coloring of $G_k$. (Note that Proposition~\ref{prop-leq-k+1} implies that at least one such clique-coloring exists.) To show that $\chi_C(G_k) = k+1$, it suffices to show that our $(k+1)$-clique-coloring of $G_k$ uses all $k+1$ colors, and so no $k$-clique-coloring of $G_k$ exists. In fact, we prove something stronger. Given $i \in \{0,\ldots,k\}$, an {\em $(i+1)$-uniform clique} of $H_i$ is an $n_i$-vertex clique $C$ of $H_i$ such that our $(k+1)$-clique-coloring of $G_k$ uses exactly $i+1$ colors on $C$, and furthermore, each of those $i+1$ colors is used on precisely $\frac{n_i}{i+1}$ vertices of $C$. (We remind the reader that $(k+1)!$ divides $n_i$, and so $\frac{n_i}{i+1}$ is an integer.) Our goal is to prove the following proposition, which immediately implies that $\chi_C(G_k) = k+1$. 

\begin{prop} \label{prop-geq-k+1} For every $i \in \{0,\ldots,k\}$, there is an $(i+1)$-uniform clique of $H_i$. 
\end{prop} 
The main ingredient of the proof of Proposition~\ref{prop-geq-k+1} is the following lemma, whose (probabilistic) proof we postpone to the end of this section. 

\begin{lem}
\label{lem:proba}
 Let $i$ and $n$ be positive integers such that $i(i+1)$ divides $n$. Let $C$ be a set of size ${n^2 \choose n}$, and let $(C_1,\ldots,C_i)$ be a partition of $C$ into $i$ equal-sized subsets. Then there exists a bijection $\phi:C \rightarrow {[n^2] \choose n}$ such that:
\begin{itemize}
\item[(1)] for any $j\in [i]$ and any $A \in {[n^2] \choose 2n}$, some member of $\phi[C_j]$ is a subset of $A$.
\item[(2)] for any $j\in [i]$ and any $B\in {[n^2] \choose n/(i+1)}$, $B$ is a subset of at least $\frac{n}{i+1}$ members of $\phi[C_j]$. 

 
\end{itemize}
\end{lem}

\begin{proof}[Proof of Proposition~\ref{prop-geq-k+1} (using Lemma~\ref{lem:proba}).] 
We proceed by induction on $i$. For $i = 0$, we observe that a $1$-uniform clique of $H_0$ is any monochromatic $n_0$-vertex clique of $H_0$. Such a clique exists because $H_0$ is a complete graph on $(k+1)n_0$ vertices, and our clique-coloring of $G_k$ uses at most $k+1$ colors. 

For the induction step, fix $i \in [k]$, and suppose that $C$ is an $i$-uniform clique of $H_{i-1}$. In particular then, $|C| = n_{i-1} = {n_i^2 \choose n_i}$. We also remind the reader that $(k+1)!$ divides $n_i$, and in particular, $i(i+1)$ divides $n_i$. Since $C$ is $i$-uniform, there is a partition $(C_1,\ldots,C_i)$ of $C$ into $i$ equal-sized monochromatic cliques; we may assume without loss of generality that for each $j \in [i]$, vertices of $C_j$ were colored with color $j$. Let $\phi:C \rightarrow {[n_i^2] \choose n_i}$ be the bijection whose existence is guaranteed by Lemma~\ref{lem:proba}. Let $(C,X)$ be the petal of $H_i$ associated with $\phi$. 

Now, for each $j \in [k+1]$, let $A_j$ be the set of all vertices of $X$ that received the color $j$. Suppose first that for some $j \in [i]$, $|A_j| \geq 2n_i$. Then by Lemma~\ref{lem:proba} (1), we know that there is some $c \in C_j$ such that $N_X(c) \subseteq A_j$, and so the clique $\{c\} \cup N_X(c)$ is monochromatic. But by Proposition~\ref{prop-max-clique}, $\{c\} \cup N_X(c)$ is a maximal clique of $G_k$, and so the fact that this clique is monochromatic contradicts the fact that $G_k$ was clique-colored. This implies that for all $j \in [i]$, $|A_j| \leq 2n_i-1$. Since $|X| = n_i^2$, and since $(A_1,\ldots,A_{k+1})$ forms a partition of $X$, it follows that there exists some index $j \in [k+1] \smallsetminus [i]$ such that $|A_j| \geq \frac{n_i^2-i(2n_i-1)}{k-i+1} \geq \frac{n_i(n_i-2i)}{k-i+1}$; by symmetry, we may assume that $|A_{i+1}| \geq \frac{n_i(n_i-2i)}{k-i+1}$. 

Since $k \geq 2$, $n_i \geq (k+1)!$, and $i \in [k]$, we see that $\frac{n_i-2i}{k-i+1} \geq \frac{1}{i+1}$, and so $|A_{i+1}| \geq \frac{n_i}{i+1}$. Fix $B \subseteq A_{i+1}$ such that $|B| = \frac{n_i}{i+1}$. By Lemma~\ref{lem:proba} (2), we know that for each $j \in [i]$, there are at least $\frac{n_i}{i+1}$ members of $C_j$ that are complete to $B$. For each $j \in [i]$, we let $B_j$ be an $\frac{n_i}{i+1}$-element subset of $C_j$ such that $B_j$ is complete to $B$, and we observe that $B \cup B_1 \cup \ldots \cup B_i$ is an $(i+1)$-uniform clique of $H_i$. This completes the induction. 
\end{proof} 

It now only remains to prove Lemma~\ref{lem:proba}. 

\begin{proof}[Proof of Lemma \ref{lem:proba}]
If $i = 1$, then the result is immediate; so assume that $i \geq 2$. (Note that this implies that either $n = 6$ or $n \geq 12$.) We denote by $p_1$ (resp. $p_2$) the probability that a random bijection $\phi:C \rightarrow {[n^2] \choose n}$ fails to satisfy (1) (resp. (2)) from Lemma~\ref{lem:proba}. We need to show that $p_1+p_2 < 1$. To simplify notation, we set $N:={n^2 \choose n}$. 

We first find an upper bound for $p_1$. Let $A \in {[n^2] \choose 2n}$. The set $A$ has $p:={2n \choose n}$ subsets of size $n$, and so for each $j \in [i]$, the probability that none of these subsets is the image of a vertex of $C_j$ is 
\begin{displaymath} 
\frac{{N-N/i \choose p}}{{N \choose p}} 
\end{displaymath} 
By the union bound over all possible choices of $A$ and all $j \in [i]$, we obtain 
\begin{displaymath} 
\begin{array}{ccccc} 
p_1 & \leq & i{n^2 \choose 2n} \frac{{N-N/i \choose p}}{{N \choose p}} & \leq & i {n^2 \choose 2n} (1-\frac{1}{i})^{p}
\end{array} 
\end{displaymath} 

We next find an upper bound for $p_2$. Let $B \in {[n^2] \choose n/(i+1)}$. There are $q:={n^2-n/(i+1) \choose n-n/(i+1)}$ subsets of size $n$ of $[n^2]$ that include $B$ (as a subset). For each $j \in [i]$, the probability that fewer than $n/(i+1)$ of these subsets are the image of a vertex of $C_j$ is 
\begin{displaymath} 
\sum\limits_{t=0}^{n/(i+1)-1} \frac{   {N-N/i \choose q-t}  {N/i \choose t} }  { {N \choose q} } 
\end{displaymath} 
Again, by the union bound over all possible choices of $B$ and all $j \in [i]$, we get
\begin{displaymath}
\begin{array}{rcl} 
p_2	&\leq& i {n^2 \choose n/(i+1)} \sum\limits_{t=0}^{n/(i+1)-1} \frac{   {N-N/i \choose q-t}  {N/i \choose t} }  { {N \choose q} } \\ \\
	&\leq& iN \sum\limits_{t=0}^{n/(i+1)-1}  \frac{ {N \choose q-t} {N \choose t}}{ {N \choose q}}   (1-\frac{1}{i})^{q-t}  (\frac{1}{i})^t\\ \\
	&\leq& iN(1-\frac{1}{i})^q \sum\limits_{t=0}^{n/(i+1)-1} (qN)^t \\ \\ 
	&\leq& iN(1-\frac{1}{i})^q \frac{n}{i+1}N^{2(\frac{n}{i+1}-1)} \\ \\ 
	&\leq&i N^n (1-\frac{1}{i})^{q}\\ \\
	&\leq& \Big(i{n^2 \choose 2n}(1-\frac{1}{i})^{q/n}\Big)^n 
\end{array}
\end{displaymath} 
Next, we show that $\frac{q}{n} \geq p$. Since $i \geq 2$ and $i(i+1)$ divides $n$, we have that $\frac{2n}{3} \leq n-\frac{n}{i+1} \leq n-2$. We now obtain the following: 
\begin{displaymath} 
\begin{array}{rclll}  
\frac{q}{n}	&=& \frac{1}{n}{n^2-\frac{n}{i+1} \choose n-\frac{n}{i+1}}& & \\ \\
	&\geq& \frac{1}{n} \cdot \frac{(n^2-n+1)^{n-\frac{n}{i+1}}}{(n-\frac{n}{i+1})!}& & \\ \\
	&\geq& \frac{1}{n} \cdot \frac{(n^2-2n+1)^{\frac{2n}{3}}}{(n-2)!} & & \\ \\
	&=& \frac{(n-1)^{\frac{4n}{3}+1}}{n!}& & \\ \\
\end{array}
\end{displaymath} 
If $n = 6$, then we see by direct calculation that $\frac{(n-1)^{\frac{4n}{3}+1}}{n!} \geq p$. Otherwise, we have that $n \geq 12$, in which case 
\begin{displaymath} 
\begin{array}{ccccccccc} 
(n-1)^{4/3} & \geq & (\frac{11}{12}n)^{4/3} & = & (\frac{11^4}{12^4}n)^{1/3}n & \geq & (\frac{11^4}{12^3})^{1/3}n & \geq & 2n
\end{array} 
\end{displaymath} 
and consequently, 
\begin{displaymath} 
\begin{array}{ccccc} 
\frac{(n-1)^{\frac{4n}{3}+1}}{n!} & \geq & \frac{(2n)^n}{n!} & \geq & p 
\end{array} 
\end{displaymath} 
So in either case, we have that $\frac{(n-1)^{\frac{4n}{3}+1}}{n!} \geq p$, and it follows that $\frac{q}{n} \geq p$. Using the fact that $0 < 1-\frac{1}{i} < 1$, we deduce that 
\begin{displaymath} 
\begin{array}{rcl} 
p_2 &\leq& \Big(i{n^2 \choose 2n}(1-\frac{1}{i})^p\Big)^n 
\end{array} 
\end{displaymath} 
and consequently, 
\begin{displaymath} 
\begin{array}{rcl} 
p_1+p_2 &\leq& i{n^2 \choose 2n}(1-\frac{1}{i})^p+\Big(i{n^2 \choose 2n}(1-\frac{1}{i})^p\Big)^n 
\end{array} 
\end{displaymath} 
Thus, in order to complete the proof, we need only show that $i{n^2 \choose 2n}(1-\frac{1}{i})^p < \frac{1}{2}$. 

Since $i(i+1)$ divides $n$, we know that $i \leq \sqrt{n}$. Further, it is well known that the central binomial coefficient ${2n \choose n}$ satisfies the inequality ${2n \choose n} \geq 2^{2n-1}/\sqrt{n}$; thus $p \geq 2^{2n-1}/\sqrt{n}$. Finally, it follows from elementary calculus that $0 < (1-\frac{1}{x})^x < \frac{1}{e}$ for all real numbers $x > 1$; in particular then, $0 < (1-\frac{1}{\sqrt{n}})^{\sqrt{n}} < \frac{1}{2}$. Using all this, we obtain the following: 
\begin{displaymath} 
\begin{array}{rclll} 
i{n^2 \choose 2n}(1-\frac{1}{i})^p & \leq & \sqrt{n}{n^2 \choose 2n}(1-\frac{1}{\sqrt{n}})^{2^{2n-1}/\sqrt{n}}
\\
\\
& \leq & \sqrt{n}\frac{n^{4n}}{(2n)!}(\frac{1}{2})^{2^{2n-1}/n} 
\\
\\
& \leq & n^{4n}(\frac{1}{2})^{2^n} 
\\
\\
& = & 2^{4n\log_2(n)-2^n} 
\end{array} 
\end{displaymath} 
Since $n \geq 6$, we have that $4n\log_2(n)-2^n < -1$, and consequently, $i{n^2 \choose 2n}(1-\frac{1}{i})^p < \frac{1}{2}$. This completes the argument. 
\end{proof}

\end{document}